\documentclass[12pt]{article}%
\usepackage{graphicx}
\usepackage[intlimits]{amsmath}
\usepackage{latexsym}
\usepackage{amsfonts}
\usepackage{amssymb}%
\makeatletter
\newfont{\footsc}{cmcsc10 at 8truept}
\newfont{\footbf}{cmbx10 at 8truept}
\newfont{\footrm}{cmr10 at 10truept}
\newtheorem{theorem}{Theorem}

\newtheorem{conjecture}[theorem]{Conjecture}
\newtheorem{corollary}[theorem]{Corollary}

\newenvironment{proof}[1][Proof]{\noindent{\textbf {#1}  }}  {\hfill$\Box$\bigskip}

\begin{document}

\title{\textbf{Graph functions maximized on a path} }
\author{Celso Marques da Silva Jr\thanks{ PEP-COPPE, Universidade Federal do Rio de
Janeiro, Brasil; \textit{email: celso.junior@ifrj.edu.br}} \thanks{Research
partially supported by CNPq-Brasil} \ and Vladimir Nikiforov\thanks{Department
of Mathematical Sciences, University of Memphis, Memphis TN 38152, USA;
\textit{email: vnikifrv@memphis.edu}}}
\maketitle

\begin{abstract}
Given a connected graph $G\ $of order $n$ and a nonnegative symmetric matrix
$A=\left[  a_{i,j}\right]  $ of order $n,$ define the function $F_{A}\left(
G\right)  $ as%
\[
F_{A}\left(  G\right)  =\sum_{1\leq i<j\leq n}d_{G}\left(  i,j\right)
a_{i,j},
\]
where $d_{G}\left(  i,j\right)  $ denotes the distance between the vertices
$i$ and $j$ in $G.$

In this note it is shown that $F_{A}\left(  G\right)  \leq F_{A}\left(
P\right)  \,$for some path of order $n.$ Moreover, if each row of $A$ has at
most one zero off-diagonal entry, then $F_{A}\left(  G\right)  <F_{A}\left(
P\right)  \,$for some path of order $n,$ unless $G$ itself is a path.

In particular, this result implies two conjectures of Aouchiche and Hansen:

- the spectral radius of the distance Laplacian of a connected graph $G$ of
order $n$ is maximal if and only if $G$ is a path;

- the spectral radius of the distance signless Laplacian of a connected graph
$G$ of order $n$ is maximal if and only if $G$ is a path.\bigskip

\textbf{AMS classification: }\textit{15A42; 05C50.}

\textbf{Keywords:}\textit{ distance matrix; distance Laplacian; distance
signless Laplacian; largest eigenvalue; path.}

\end{abstract}

\section{Introduction and main results}

The aim of the present note is to give a general approach to problems like the
following conjectures of Aouchiche and Hansen \cite{AuHa11,AuHa13a}:

\begin{conjecture}
\label{con1}The largest eigenvalue of the distance Laplacian of a connected
graph $G$ of order $n$ is maximal if and only if $G$ is a path.
\end{conjecture}

\begin{conjecture}
\label{con2}The largest eigenvalue of the distance signless Laplacian of a
connected graph $G$ of order $n$ is maximal if and only if $G$ is a path.
\end{conjecture}

First, let us introduce some notation and recall a few definitions. We write
$\lambda\left(  A\right)  $ for the largest eigenvalue of a symmetric matrix
$A$. Given a connected graph $G,$ let $D\left(  G\right)  $ be the distance
matrix of $G,$ and let $T\left(  G\right)  $ be the diagonal matrix of the
rowsums of $D\left(  G\right)  .$ The matrix $D^{L}\left(  G\right)  =T\left(
G\right)  -D\left(  G\right)  $ is called the \emph{distance Laplacian }of
$G,$ and the matrix $D^{Q}\left(  G\right)  =T\left(  G\right)  +D\left(
G\right)  $ is called the \emph{distance signless Laplacian} of $G$. The
matrices $D^{L}\left(  G\right)  $ and $D^{Q}\left(  G\right)  $ have been
introduced by Aouchiche and Hansen and have been intensively studied recently,
see, e.g., \cite{AuHa11,AuHa13a,AuHa13,LiLu14,NaPa14,XZL13}.\medskip

Very recently, Lin and Lu \cite{LiLu14} succeeded to prove Conjecture
\ref{con2}, but Conjecture \ref{con1} seems a bit more difficult and still
holds. Furthermore, Conjectures \ref{con1} and \ref{con2} suggest a similar
problem for the distance matrix itself. As it turns out such problem has been
partially solved a while ago by Ruzieh and Powers \cite{RuPo90}, who showed
that the largest eigenvalue of the distance matrix of a connected graph $G$ of
order $n$ is maximal if $G$ is a path. The complete solution, however, was
given more recently by Stevanovi\'{c} and Ili\'{c} \cite{StIl10}.

\begin{theorem}
[\textbf{\cite{RuPo90},\cite{StIl10}}]\label{thSI}The largest eigenvalue of
the distance matrix of a connected graph $G$ of order $n$ is maximal if and
only if $G$ is a path.
\end{theorem}

These result are believed to belong to spectral graph theory, and their proofs
involve nonnegligible amount of calculations. Our goal is to show that all
these results stem from a much more general assertion that has nothing to do
with eigenvalues. To this end, we shall introduce a fairly general graph
function and shall study its maxima.

\subsection{The function $F_{A}\left(  G\right)  $ and its maxima}

Let $G$ be a connected graph of order $n.$ Write $d_{G}\left(  i,j\right)  $
for the distance between the vertices $i$ and $j$ in $G,$ and let $A=\left[
a_{i,j}\right]  $ be a nonnegative symmetric matrix of order $n.$ Define the
function $F_{A}\left(  G\right)  $ as%
\[
F_{A}\left(  G\right)  =\sum_{1\leq i<j\leq n}d_{G}\left(  i,j\right)
a_{i,j}.
\]
Clearly $d_{G}\left(  i,i\right)  =0$ for any $i\in V\left(  G\right)  ,$ so
the diagonal of $A$ is irrelevant for $F_{A}\left(  G\right)  $.

In fact, the function $F_{A}\left(  G\right)  $ is quite mainstream, as it can
be represented as%
\[
F_{A}\left(  G\right)  =\left\Vert A\circ D\left(  G\right)  \right\Vert
_{l_{1}},
\]
where $\circ$ denotes the entrywise Hadamard product of matrices, and
$\left\Vert \cdot\right\Vert _{l_{1}}$ is the $l_{1}$ norm. This viewpoint
suggests a number of extensions, which we shall investigate elsewhere.\medskip

Next, we focus on the extremal points of $F_{A}\left(  G\right)  ,$ that is to
say, we want to know which connected graphs $G$ of order $n$ satisfy the
condition%
\[
F_{A}\left(  G\right)  =\max\left\{  F_{A}\left(  H\right)  :\text{ }H\text{
is a connected graph of order }n\right\}  .\text{ }%
\]
In particular, we prove the somewhat surprising fact that for any admissible
matrix $A,$ the function $F_{A}\left(  G\right)  $ is always maximized by a
path. More precisely the following theorem holds.

\begin{theorem}
\label{thg}Let $G$ be a connected graph of order $n$ and let $A=\left[
a_{i,j}\right]  $ be a symmetric matrix of order $n.$ If $A$ is nonnegative,
then there is a path $P$ with $V\left(  P\right)  =V\left(  G\right)  $\ such
that
\begin{equation}
\sum_{1\leq i<j\leq n}d_{G}\left(  i,j\right)  a_{i,j}\leq\sum_{1\leq i<j\leq
n}d_{P}\left(  i,j\right)  a_{i,j}. \label{mi}%
\end{equation}

\end{theorem}

It is not hard to find nonnegative symmetric matrices $A$ for which
$F_{A}\left(  G\right)  $ is maximized also by graphs other than paths. Thus,
it is natural to attempt to characterize all symmetric, nonnegative matrices
$A,$ for which $F_{A}\left(  G\right)  $ is maximal only if $G$ is a path. The
complete solution of this problem seems difficult, so we shall give only a
partial solution, sufficient for our goals.

\begin{theorem}
\label{ths}Let $G$ be a connected graph of order $n$ and let $A=\left[
a_{i,j}\right]  $ be a symmetric nonnegative matrix of order $n.$ If each row
of $A$ has at most one zero off-diagonal entry, and $G$ is not a path, then
there is a path $P$ with $V\left(  P\right)  =V\left(  G\right)  $\ such that%
\begin{equation}
\sum_{1\leq i<j\leq n}d_{G}\left(  i,j\right)  a_{i,j}<\sum_{1\leq i<j\leq
n}d_{P}\left(  i,j\right)  a_{i,j}. \label{mip}%
\end{equation}

\end{theorem}

As yet we know of no application that exploits the full strength of Theorem
\ref{ths}. Indeed, to prove Conjectures \ref{con1} and \ref{con2}, and Theorem
\ref{thSI}, we shall use only the following simple corollary.

\begin{corollary}
\label{cor1} Let $G$ be a connected graph of order $n$ and let $A=\left[
a_{i,j}\right]  $ be a symmetric matrix of order $n.$ If each off-diagonal
entry of $A$ is positive, and $G$ is not a path, then there is a path $P$ with
$V\left(  P\right)  =V\left(  G\right)  $\ such that $F_{A}\left(  P\right)
>F_{A}\left(  G\right)  .$
\end{corollary}

\subsection{Proofs of Conjectures \ref{con1} and \ref{con2}, and Theorem
\ref{thSI}}

We proceed with the proof of Conjecture \ref{con2}. Let $G$ be a connected
graph of order $n$ for which $\lambda\left(  D^{Q}\left(  G\right)  \right)  $
is maximal within all connected graphs of order $n$. We shall prove that $G$
is a path. Let $\mathbf{x}=\left(  x_{1},\ldots,x_{n}\right)  $ be a unit
eigenvector to $\lambda\left(  D^{Q}\left(  G\right)  \right)  .$ Since
$D^{Q}\left(  G\right)  $ is irreducible, the vector $\mathbf{x}$ is positive.
Define an $n\times n$ matrix $A=\left[  a_{i,j}\right]  $ by letting
$a_{i,j}=\left(  x_{i}+x_{j}\right)  ^{2}.$ Clearly $A$ is symmetric and
nonnegative. As is well-known,
\[
\lambda\left(  D^{Q}\left(  G\right)  \right)  =\left\langle D^{Q}\left(
G\right)  \mathbf{x},\mathbf{x}\right\rangle =\sum_{1\leq i<j\leq n}%
d_{G}\left(  i,j\right)  \left(  x_{i}+x_{j}\right)  ^{2}=F_{A}\left(
G\right)  .
\]
Since each off-diagonal entry of $A$ is positive, Corollary \ref{cor1} implies
that either $G=P_{n}$ or there is a path $P$ with $V\left(  P\right)
=V\left(  G\right)  $\ such that $F_{A}\left(  P\right)  >F_{A}\left(
G\right)  .$ The latter cannot hold as we would have
\[
\lambda\left(  D^{Q}\left(  G\right)  \right)  =F_{A}\left(  G\right)
<F_{A}\left(  P\right)  \leq\lambda\left(  D^{Q}\left(  P\right)  \right)  ,
\]
contrary to the choice of $G.$ Hence $G=P_{n},$ completing the proof of
Conjecture \ref{con2}.

Theorem \ref{thSI} can be proved in the same way, with $A=\left[
a_{i,j}\right]  $ defined by $a_{i,j}=x_{i}x_{j}.$ However, Conjecture
\ref{con1} requires a slightly more careful approach.

Let $G$ be a connected graph of order $n$ such that $\lambda\left(
D^{L}\left(  G\right)  \right)  $ is maximal among all connected $n$ vertex
graphs. We shall prove that $G$ must be a path. Let $\mathbf{x}=\left(
x_{1},\ldots,x_{n}\right)  $ be a unit eigenvector to $\lambda\left(
D^{L}\left(  G\right)  \right)  $ and define an $n\times n$ matrix $A=\left[
a_{i,j}\right]  $ by letting $a_{i,j}=\left(  x_{i}-x_{j}\right)  ^{2}.$
Clearly $A$ is symmetric and nonnegative. Also, it is well-known that
\[
\lambda\left(  D^{L}\left(  G\right)  \right)  =\left\langle D^{L}\left(
G\right)  \mathbf{x},\mathbf{x}\right\rangle =\sum_{1\leq i<j\leq n}%
d_{G}\left(  i,j\right)  \left(  x_{i}-x_{j}\right)  ^{2}=F_{A}\left(
G\right)  .
\]
However, at this stage we cannot rule out that $A$ has numerous zero entries,
and so Corollary \ref{cor1} does not apply as before. Yet Theorem \ref{thg}
implies that there is a path $P$ with $V\left(  P\right)  =V\left(  G\right)
$\ such that $F_{A}\left(  P\right)  \geq F_{A}\left(  G\right)  ;$ hence,%
\[
\lambda\left(  D^{L}\left(  G\right)  \right)  =F_{A}\left(  G\right)  \leq
F_{A}\left(  P\right)  \leq\lambda\left(  D^{L}\left(  P\right)  \right)  .
\]
Due to the choice of $G,$ equalities should hold throughout the above line,
implying that $\mathbf{x}$ is an eigenvector to $P_{n}.$ But in Theorems. 4.4
and 4.6 of \cite{NaPa14} Nath and Paul have established that all entries of an
eigenvector to $\lambda\left(  D^{L}\left(  P\right)  \right)  $ are different
and so the off-diagonal entries of $A$ are positive. Now we apply Corollary
\ref{cor1} and finish the proof as for Conjecture \ref{con2}.

\section{Proofs of the main theorems}

For graph notation undefined here we refer the reader to \cite{Bol98}. For
general properties of the distance Laplacian and the distance signless
Laplacian the reader is referred to \cite{AuHa11,AuHa13a,AuHa13}.

Here is some notation that will be used later in the proofs:

- $P_{n}$ and $C_{n}$ stand for the path and cycle of order $n;$

- $G-u$ denotes the graph obtained from $G$ by removing the vertex $u;$

- $G-\left\{  u,v\right\}  $ denotes the graph obtained from $G$ by removing
the vertices $u$ and $v.\medskip$

We shall assume that any graph of order $n$ is defined on the vertex set
$\left[  n\right]  =\left\{  1,\ldots,n\right\}  .\medskip$

The proofs of Theorems \ref{thg} and \ref{ths} have the same general
structure, but the latter requires a lot of extra details so it will be
presented separately. \medskip

\begin{proof}
[\textbf{Proof of Theorem \ref{thg}}]Note first that if $H$ is a spanning tree
of $G,$ then $d_{G}\left(  i,j\right)  \leq d_{H}\left(  i,j\right)  $ for
every $i,j\in V\left(  G\right)  ;$ hence
\[
\sum_{1\leq i<j\leq n}d_{G}\left(  i,j\right)  a_{i,j}\leq\sum_{1\leq i<j\leq
n}d_{H}\left(  i,j\right)  a_{i,j}.
\]
Therefore, we may and shall assume that $G$ is a tree itself. We carry out the
proof by induction on $n.$ If $n\leq3,$ every tree of order $n$ is a path, so
there is nothing to prove in this case. Assume now that $n>3$ and the
assertion holds for any $n^{\prime}$ such that $n^{\prime}<n.$ Choose a vertex
$u\in V\left(  G\right)  $ of degree $1.$ By symmetry, we assume that $u=n,$
and let $k$ be the single neighbor of $u;$ hence $G-n$ is a tree of order
$n-1.$

Define a symmetric matrix $A^{\prime}=\left[  a_{ij}^{\prime}\right]  $ of
order $n-1$ as follows:%
\[
a_{i,j}^{\prime}=\left\{
\begin{array}
[c]{ll}%
a_{i,j}, & \text{if\textbf{ }}i\neq k\text{ and }j\neq k;\text{\textbf{ } }\\
a_{k,j}+a_{n,j}, & \text{if\textbf{ }}i=k;\\
a_{i,k}+a_{i,n}, & \text{if\textbf{ }}j=k.
\end{array}
\right.
\]
Clearly $A^{\prime}$ is a symmetric nonnegative matrix. By the induction
assumption there is a path $P^{\prime}$ with $V\left(  P^{\prime}\right)
=V\left(  G-n\right)  =\left[  n-1\right]  $ such that
\begin{equation}
\sum_{1\leq i<j<n}d_{G-n}\left(  i,j\right)  a_{i,j}^{\prime}\leq\sum_{1\leq
i<j<n}d_{P^{\prime}}\left(  i,j\right)  a_{i,j}^{\prime}. \label{in1}%
\end{equation}
On the other hand, for each $j\in V\left(  G-n\right)  ,$ the shortest path
between $n$ and $j$ contains $k,$ so%
\[
d_{G}\left(  j,n\right)  =d_{G-n}\left(  j,k\right)  +1.
\]
Hence we see that
\begin{align*}
\sum_{1\leq i<j\leq n}d_{G}\left(  i,j\right)  a_{i,j}  &  =\sum_{j=1}%
^{n-1}d_{G}\left(  j,n\right)  a_{j,n}+\sum_{1\leq i<j<n}d_{G-n}\left(
i,j\right)  a_{i,j}\\
&  =\sum_{j=1}^{n-1}\left(  d_{G-n}\left(  k,j\right)  +1\right)  a_{j,n}%
+\sum_{1\leq i<j<n}d_{G-n}\left(  i,j\right)  a_{i,j}\\
&  =\sum_{j=1}^{n-1}a_{n,j}+\sum_{1\leq i<j<n}d_{G-n}\left(  i,j\right)
a_{i,j}^{\prime}.
\end{align*}
Now, (\ref{in1}) implies that
\begin{equation}
\sum_{1\leq i<j\leq n}d_{G}\left(  i,j\right)  a_{i,j}\leq\sum_{j=1}%
^{n-1}a_{j,n}+\sum_{1\leq i<j<n}d_{P^{\prime}}\left(  i,j\right)
a_{i,j}^{\prime}. \label{in2}%
\end{equation}
Further, write $T$ for the tree obtained form the path $P^{\prime}$ by joining
$n$ to the vertex $k\in V\left(  P^{\prime}\right)  $. As before, we see that
\begin{align*}
\sum_{1\leq i<j\leq n}d_{T}\left(  i,j\right)  a_{i,j}  &  =\sum_{j=1}%
^{n-1}d_{T}\left(  j,n\right)  a_{j,n}+\sum_{1\leq i<j<n}d_{T-n}\left(
i,j\right)  a_{i,j}\\
&  =\sum_{j=1}^{n-1}\left(  d_{P^{\prime}}\left(  j,k\right)  +1\right)
a_{j,n}+\sum_{1\leq i<j<n}d_{P^{\prime}}\left(  i,j\right)  a_{i,j}\\
&  =\sum_{j=1}^{n-1}a_{j,n}+\sum_{1\leq i<j<n}d_{P^{\prime}}\left(
i,j\right)  a_{i,j}^{\prime}.
\end{align*}
Hence, (\ref{in2}) implies that%
\[
\sum_{1\leq i<j\leq n}d_{G}\left(  i,j\right)  a_{i,j}\leq\sum_{1\leq i<j\leq
n}d_{T}\left(  i,j\right)  a_{i,j}.
\]
If $T=P_{n}$, there is nothing to prove, so suppose that $T\neq P_{n}.$ To
complete the proof we shall show that we can join $n$ to one of the ends of
$P^{\prime}$ so that $F_{A}\left(  T\right)  $ will not decrease.

By symmetry, assume that the vertex sequence of the path $P^{\prime}$ is
precisely $1,2,\ldots,n-1;$ thus the neighbor $k$ of $n$ satisfies $1<k<n-1.$
Write $A_{0}$ for the principal submatrix of $A$ in the first $n-1$ rows and
note that%
\[
F_{A}\left(  T\right)  =\sum_{i=1}^{k}\left(  k-i+1\right)  a_{i,n}%
+\sum_{i=k+1}^{n-1}\left(  i-k+1\right)  a_{i,n}+F_{A_{0}}\left(  P^{\prime
}\right)  .
\]
Next, delete the edge $\left\{  n,k\right\}  $ in $T,$ add the edge $\left\{
n,1\right\}  ,$ and write $T_{1}$ for the resulting path. If $F_{A}\left(
T_{1}\right)  >F_{A}\left(  T\right)  ,$ the proof is completed, so let us
assume that $F_{A}\left(  T_{1}\right)  \leq F_{A}\left(  T\right)  .$ Since
\[
F_{A}\left(  T_{1}\right)  =\sum_{i=1}^{n-1}ia_{i,n}+F_{A_{0}}\left(
P^{\prime}\right)  ,
\]
we see that,
\[
\sum_{i=1}^{k-1}\left(  k-i+1\right)  a_{i,n}+\sum_{i=k}^{n-1}\left(
i-k+1\right)  a_{i,n}\geq\sum_{i=1}^{n-1}ia_{i,n}%
\]
and so%
\[
\sum_{i=1}^{k-1}\left(  k-2i+1\right)  a_{i,n}\geq\left(  k-1\right)  \left(
a_{k,n}+\cdots+a_{n-1,n}\right)  .
\]
Hence,%
\begin{equation}
\left(  k-1\right)  \left(  a_{1,n}+\cdots+a_{k-1,n}\right)  \geq\left(
k-1\right)  \left(  a_{k,n}+\cdots+a_{n-1,n}\right)  . \label{in3}%
\end{equation}

Now, delete the edge $\left\{  n,k\right\}  $ in $T,$ add the edge $\left\{
n,n-1\right\}  ,$ and write $T_{2}$ for the resulting path. If $F_{A}\left(
T_{2}\right)  >F_{A}\left(  T\right)  ,$ the proof is completed, so let us
assume that $F_{A}\left(  T_{2}\right)  \leq F_{A}\left(  T\right)  .$ Since
\[
F_{A}\left(  T_{2}\right)  =\sum_{i=1}^{n-1}\left(  n-i\right)  a_{i,n}%
+F_{A_{0}}\left(  P^{\prime}\right)  ,
\]
we see that
\[
\sum_{i=1}^{k-1}\left(  k-i+1\right)  a_{i,n}+\sum_{i=k}^{n-1}\left(
i-k+1\right)  a_{i,n}\geq\sum_{i=1}^{n-1}\left(  n-i\right)  a_{i,n},
\]
and so%
\[
\sum_{i=k}^{n-1}\left(  2i-k-n+1\right)  a_{i,n}\geq\left(  n-k-1\right)
\left(  a_{1,n}+\cdots+a_{k-1,n}\right)  .
\]
Hence,
\[
\left(  n-k-1\right)  \left(  a_{k,n}+\cdots+a_{n-1,n}\right)  \geq\left(
n-k-1\right)  \left(  a_{1,n}+\cdots+a_{k-1,n}\right)  .
\]
This inequality, together with (\ref{in3}), implies that
\[
a_{k,n}+\cdots+a_{n-1,n}=a_{1,n}+\cdots+a_{k-1,n},
\]
and that $F_{A}\left(  T_{1}\right)  =F_{A}\left(  T\right)  $ and
$F_{A}\left(  T_{2}\right)  =F_{A}\left(  T\right)  .$ This completes the
induction step and the proof of Theorem \ref{thg}.
\end{proof}

\subsection{Proof of Theorem \ref{ths}}

Most of the proof of Theorem \ref{ths} deals with the case of $G$ being a
tree, so we extract this part in Theorem \ref{tht} below. The general case
will be deduced later by different means.

For convenience write $N\left(  n\right)  $ for the class of all symmetric
nonnegative matrix of order $n$ such that each row of $A$ has at most one zero
off-diagonal entry.

\begin{theorem}
\label{tht}Let $G$ be a tree of order $n.$ If $A\in N\left(  n\right)  $ and
$G\neq P_{n},$ then there exists a path $P$ with $V\left(  P\right)  =V\left(
G\right)  $\ such that $F_{A}\left(  G\right)  <F_{A}\left(  P\right)  .$
\end{theorem}

\begin{proof}
Our proof is by induction on $n$ and is structured as the proof of Theorem
\ref{thg}. If $n\leq3,$ every tree of order $n$ is a path, so there is nothing
to prove in this case. For technical reason we would like to give a direct
proof for $n=4$ as well. There are two trees of order $4$ - a path and a star.
Assume that $G$ is a star, and by symmetry suppose that $2$ is its center. We
have
\[
F_{A}\left(  G\right)  =2a_{4,1}+a_{4,2}+2a_{4,3}+a_{1,2}+a_{2,3}+2a_{1,3}.
\]
Remove the edge $\left\{  4,2\right\}  $ and add the edge $\left\{
4,1\right\}  ,$ thus obtaining a path $G_{1}.$ Assume for a contradiction that
$F_{A}\left(  G\right)  \geq F_{A}\left(  G_{1}\right)  ,$ which implies that
$a_{4,1}\geq a_{4,2}+a_{4,3}.$ Now, remove from $G$ the edge $\left\{
4,2\right\}  $ and add the edge $\left\{  4,3\right\}  ,$ thus obtaining a
path $G_{2}.$ Assume for a contradiction that $F_{A}\left(  G\right)  \geq
F_{A}\left(  G_{2}\right)  ,$ which implies that $a_{4,3}\geq a_{4,2}%
+a_{4,1}.$ We conclude that $a_{4,2}=0.$ By symmetry, we also get $a_{1,2}=0$
and $a_{3,2}=0;$ hence $A$ has a zero row, contradicting the hypothesis. Thus,
$G$ is a path.

Assume now that $n\geq5$ and the assertion of Theorem holds for any
$n^{\prime}$ such that $n^{\prime}<n.$ Let $G$ be tree for which $F_{A}\left(
G\right)  $ attains a maximum. We shall prove that $G=P_{n}.$ Choose a vertex
$u\in V\left(  G\right)  $ of degree $1.$ By symmetry, we assume that $u=n,$
and let $k$ be the single neighbor of $u$; hence $G-n$ is a tree of order
$n-1.$

Define a symmetric matrix $A^{\prime}=\left[  a_{ij}^{\prime}\right]  $ of
order $n-1$ as follows%
\[
a_{i,j}^{\prime}=\left\{
\begin{array}
[c]{ll}%
a_{i,j}, & \text{if\textbf{ }}i\neq k\text{ and }j\neq k;\text{\textbf{ } }\\
a_{k,j}+a_{n,j}, & \text{if\textbf{ }}i=k;\\
a_{i,k}+a_{i,n}, & \text{if\textbf{ }}j=k.
\end{array}
\right.
\]
Clearly $A^{\prime}\in N\left(  n-1\right)  $. Suppose that $G-n\neq P_{n-1}.$
By the induction assumption, there is a path $P^{\prime}$ with $V\left(
P^{\prime}\right)  =V\left(  G-n\right)  =\left[  n-1\right]  $ such that
\[
\sum_{1\leq i<j<n}d_{G-n}\left(  i,j\right)  a_{i,j}^{\prime}<\sum_{1\leq
i<j<n}d_{P^{\prime}}\left(  i,j\right)  a_{i,j}^{\prime}.
\]
Hence, as in the proof of Theorem \ref{thg}, we find that
\[
F_{A}\left(  G\right)  =\sum_{j=1}^{n-1}a_{j,n}+\sum_{1\leq i<j<n}%
d_{G-n}\left(  i,j\right)  a_{i,j}^{\prime}<\sum_{j=1}^{n-1}a_{j,n}%
+\sum_{1\leq i<j<n}d_{P^{\prime}}\left(  i,j\right)  a_{i,j}^{\prime}.
\]
Now, join $n$ to $k,$ and write $T$ for the obtained tree. As before, we see
that
\[
F_{A}\left(  T\right)  =\sum_{j=1}^{n-1}a_{j,n}+\sum_{1\leq i<j<n}%
d_{P^{\prime}}\left(  i,j\right)  a_{i,j}^{\prime}>F_{A}\left(  G\right)  .
\]
This contradicts the assumption that $F_{A}\left(  G\right)  $ is maximal.
Therefore $G-n=P_{n-1}.$

By symmetry, assume that the vertex sequence of the path $G-n$ is precisely
$1,2,\ldots,n-1.$ If $k=1$ or $k=n-1,$ we see that $G=P_{n},$ so let us assume
that $1<k<n-1$. To complete the proof we shall show that we can join $n$ to
$1$ or to $n-1$ so that $F_{A}\left(  G\right)  $ will increase.

Write $A_{0}$ for the principal submatrix of $A$ in the first $n-1$ rows and
note that%
\[
F_{A}\left(  G\right)  =\sum_{i=1}^{k}\left(  k-i+1\right)  a_{i,n}%
+\sum_{i=k+1}^{n-1}\left(  i-k+1\right)  a_{i,n}+F_{A_{0}}\left(  G-n\right)
.
\]
Next, delete the edge $\left\{  n,k\right\}  $ in $G,$ add the edge $\left\{
n,1\right\}  ,$ and write $G_{1}$ for the resulting path. Since $F_{A}\left(
G\right)  $ is maximal, we see that $F_{A}\left(  G_{1}\right)  \leq
F_{A}\left(  G\right)  .$ From
\[
F_{A}\left(  G_{1}\right)  =\sum_{i=1}^{n-1}ia_{i,n}+F_{G-n}\left(
A_{0}\right)
\]
it follows that,
\[
\sum_{i=1}^{k-1}\left(  k-i+1\right)  a_{i,n}+\sum_{i=k}^{n-1}\left(
i-k+1\right)  a_{i,n}\geq\sum_{i=1}^{n-1}ia_{i,n},
\]
and so%
\[
\sum_{i=1}^{k-1}\left(  k-2i+1\right)  a_{i,n}\geq\left(  k-1\right)  \left(
a_{k,n}+\cdots+a_{n-1,n}\right)  .
\]
Hence, letting%
\[
S_{1}=-2\sum_{i=1}^{k-1}\left(  i-1\right)  a_{i,n}%
\]
we see that
\begin{equation}
\left(  k-1\right)  \left(  a_{1,n}+\cdots+a_{k-1,n}\right)  +S_{1}\geq\left(
k-1\right)  \left(  a_{k,n}+\cdots+a_{n-1,n}\right)  \label{in4}%
\end{equation}

Finally, delete the edge $\left\{  n,k\right\}  $ in $G,$ add the edge
$\left\{  n,n-1\right\}  ,$ and write $G_{2}$ for the resulting path. Since
$F_{A}\left(  G\right)  $ is maximal, we see that $F_{A}\left(  G_{2}\right)
\leq F_{A}\left(  G\right)  .$ From
\[
F_{A}\left(  G_{2}\right)  =\sum_{i=1}^{n-1}\left(  n-i\right)  a_{i,n}%
+F_{G-n}\left(  A_{0}\right)
\]
it follows that
\[
\sum_{i=1}^{k}\left(  k-i+1\right)  a_{i,n}+\sum_{i=k+1}^{n-1}\left(
i-k+1\right)  a_{i,n}\geq\sum_{i=1}^{n-1}\left(  n-i\right)  a_{i,n},
\]
and so
\[
\sum_{i=k}^{n-1}\left(  2i-k-n+1\right)  a_{i,n}\geq\left(  n-k-1\right)
\left(  a_{1,n}+\cdots+a_{k-1,n}\right)  .
\]
Hence, letting%
\[
S_{2}=-2\sum_{i=k}^{n-1}\left(  n-i-1\right)  a_{i,n}%
\]%
\[
\left(  n-k-1\right)  \left(  a_{k,n}+\cdots+a_{n-1,n}\right)  +S_{2}%
\geq\left(  n-k-1\right)  \left(  a_{1,n}+\cdots+a_{k-1,n}\right)  .
\]
Comparing this inequality with (\ref{in4}), in view of $S_{1}\leq0$ and
$S_{2}\leq0,$ we find that
\[
a_{k,n}+\cdots+a_{n-1,n}=a_{1,n}+\cdots+a_{k-1,n}\text{ \ \ and \ \ \ }%
S_{1}=S_{2}=0.
\]
Hence,%
\[
a_{2,n}=\cdots=a_{k-1,n}=0\text{ \ \ and \ \ }a_{k,n}=\cdots=a_{n-2,n}=0.
\]
Since $n-3\geq2$, among the off-diagonal entries of the $n$'th row of $A,$
there are two that are zero, contrary to the hypothesis. Therefore, $G=P_{n},$
completing the induction step and the proof of Theorem \ref{tht}.
\end{proof}

Armed with Theorem \ref{tht}, we are able the complete the proof of Theorem
\ref{ths}.\medskip

\begin{proof}
[\textbf{Proof of Theorem \ref{ths}}]First we shall prove Theorem \ref{ths} if
$G$ is a unicyclic graph, i.e., if $G$ has exactly $n$ edges. Thus, let $G$ be
a connected unicyclic graph of order $n\geq3.$ It is known that $G\ $contains
a single cycle. If $G$ is not the cycle $C_{n\text{ }}$itself, then
$G\ $contains a spanning tree $H\ $with maximum degree $\Delta\left(
H\right)  \geq3;$ thus $H\neq P_{n}.$ Hence, Theorem \ref{tht} implies that
there is a path $P$ with $V\left(  P\right)  =V\left(  G\right)  $ such that
\[
F_{A}\left(  G\right)  =\sum_{1\leq i<j\leq n}d_{G}\left(  i,j\right)
a_{i,j}\leq\sum_{1\leq i<j\leq n}d_{H}\left(  i,j\right)  a_{i,j}<\sum_{1\leq
i<j\leq n}d_{P}\left(  i,j\right)  a_{i,j}.
\]
If $G$ is the cycle $C_{n\text{ }}$itself, let $i,j,k$ be three consecutive
vertices along the cycle. The removal of the edge $\left\{  i,j\right\}  $
increases the distance between $i$ \ and $j,$ i.e.,
\[
d_{G}\left(  i,j\right)  <d_{G-\left\{  i,j\right\}  }\left(  i,j\right)
\]
\ and on the other hand
\[
F_{A}\left(  G\right)  \leq F_{A}\left(  G-\left\{  i,j\right\}  \right)  .
\]
If $F_{A}\left(  G\right)  <F_{A}\left(  G-\left\{  i,j\right\}  \right)  ,$
the theorem is proved, otherwise $F_{A}\left(  G\right)  =F_{A}\left(
G-\left\{  i,j\right\}  \right)  $ and so $a_{i,j}=0.$ By the same token we
obtain $a_{j,k}=0;$ hence among the off-diagonal entries of the $k$'th row of
$A$ there are two that are zero, contrary to the hypothesis. So the theorem
holds for unicyclic graphs.

Finally, note that any connected graph $G$ that is not a tree contains a
connected unicyclic spanning subgraph $H$ or is unicyclic itself. Hence, if
$G$ is not a tree, then $F_{A}\left(  G\right)  \leq F_{A}\left(  H\right)  $
for some connected unicyclic $H,$ and thus there is a path $P$ with $V\left(
P\right)  =V\left(  G\right)  $ such that
\[
F_{A}\left(  G\right)  \leq F_{A}\left(  H\right)  <F\left(  P\right)  .
\]
The proof of Theorem \ref{ths} is completed.
\end{proof}

\section{Concluding remarks}

Results similar to Theorem \ref{thSI} have been known for the adjacency
matrix, the Laplacian, and the signless Laplacian of a connected graph $G:$

\begin{theorem}
[\cite{LoPe73}]\label{t1}The largest eigenvalue of the adjacency matrix of a
connected graph $G$ of order $n$ is minimal if and only if $G$ is a path.
\end{theorem}

\begin{theorem}
[\cite{PeGu02}]\label{t2}The largest eigenvalue of the Laplacian of a
connected graph $G$ of order $n$ is minimal if and only if $G$ is a path.
\end{theorem}

\begin{theorem}
[\cite{Yan02}]\label{t3}The largest eigenvalue of the signless Laplacian of a
connected graph $G$ of order $n$ is minimal if and only if $G$ is a path.
\end{theorem}

In the light of the present note we would like to raise the following
question:\medskip

\textbf{Question. }\emph{Is there a result similar to Theorem }\ref{ths}\emph{
that implies Theorems} \ref{t1}, \ref{t2},\emph{ and }\ref{t3}.

\end{document}